\documentclass[12pt]{amsart}

\usepackage{amsmath,amssymb,amsthm,amsfonts}
\usepackage{graphicx}
\usepackage{hyperref}

\theoremstyle{plain}
\newtheorem{thm}{Theorem}
\newtheorem{lem}{Lemma}
\newtheorem{cor}{Corollary}

\theoremstyle{definition}

\theoremstyle{remark}

\def\dist{\mathrm {dist}\, } 
 
\def\Per{{\rm Per}\, }
\def\PerM{{\rm Per}_M}
\def\RR{\mathbb R}
\def\Vol{{\rm Vol}\, }

\begin{document}

\title[Perimeter under multiple Steiner symmetrizations]
{Perimeter under multiple\\Steiner symmetrizations}

\author{Almut Burchard and Gregory R. Chambers}
\address{University of Toronto, Department of Mathematics,
40 St. George Street, Toronto, Canada M5S 2E4.
{\tt \{almut,chambers\}@$\,$math.toronto.edu}}
\date{September 19, 2012}
\begin{abstract}
Steiner symmetrization along $n$ linearly independent directions
transforms every compact subset of $\RR^n$
into a set of finite perimeter.
\end{abstract} \maketitle

Steiner symmetrization is a volume-preserving rearrangement 
that introduces a hyperplane of symmetry.
A key property is that Steiner symmetrization
strictly reduces the perimeter of every convex set 
that is not already reflection symmetric~\cite{S}. 
The perimeter of a non-convex set of finite 
perimeter decreases strictly under Steiner 
symmetrization in most directions, but not necessarily
in all of them~\cite{CCF}. 

We seek to bound the perimeter of an arbitrary compact set 
$A\subset \RR^n$ after a finite sequence of Steiner 
symmetrizations.  Our main result is that $n$ consecutive
Steiner symmetrization in linearly independent directions suffice 
to transform $A$ into a set of finite perimeter. 

\begin{thm} [Perimeter estimate] \label{thm:per} 
If $A\subset\RR^n$ is a compact set and $u_1,\dots, u_n$ 
are linearly independent unit vectors in $\RR^n$, then
\begin{equation}\label{eq:per}
\Per(S_{u_n}\dots S_{u_1}A) \ \le\ 
\frac{a_nR^{n-1}}{|\det\,(u_1,\dots, u_n)\,|}\,,
\end{equation}
where $a_n=2n\omega_{n-1}$, and $R$ is the outradius of $A$.
\end{thm}

The theorem is motivated by the special case
of the coordinate directions $e_1,\dots, e_n$. 
The set $S_{e_n}\dots S_{e_1}A$
is symmetric under reflection at each coordinate
hyperplane, and its intersection with the positive cone
lies under the graph of a monotone function
$x_n=f(x_1,\dots, x_{n-1})$.  The perimeter of such a set
is bounded by twice the sum of the area of 
its projections onto the coordinate hyperplanes, 
which cannot exceed $a_nR^{n-1}$,
see Figure~\ref{fig:ortho}.

\begin{figure}[ht]
\centerline{
\includegraphics[width=1.0\textwidth]{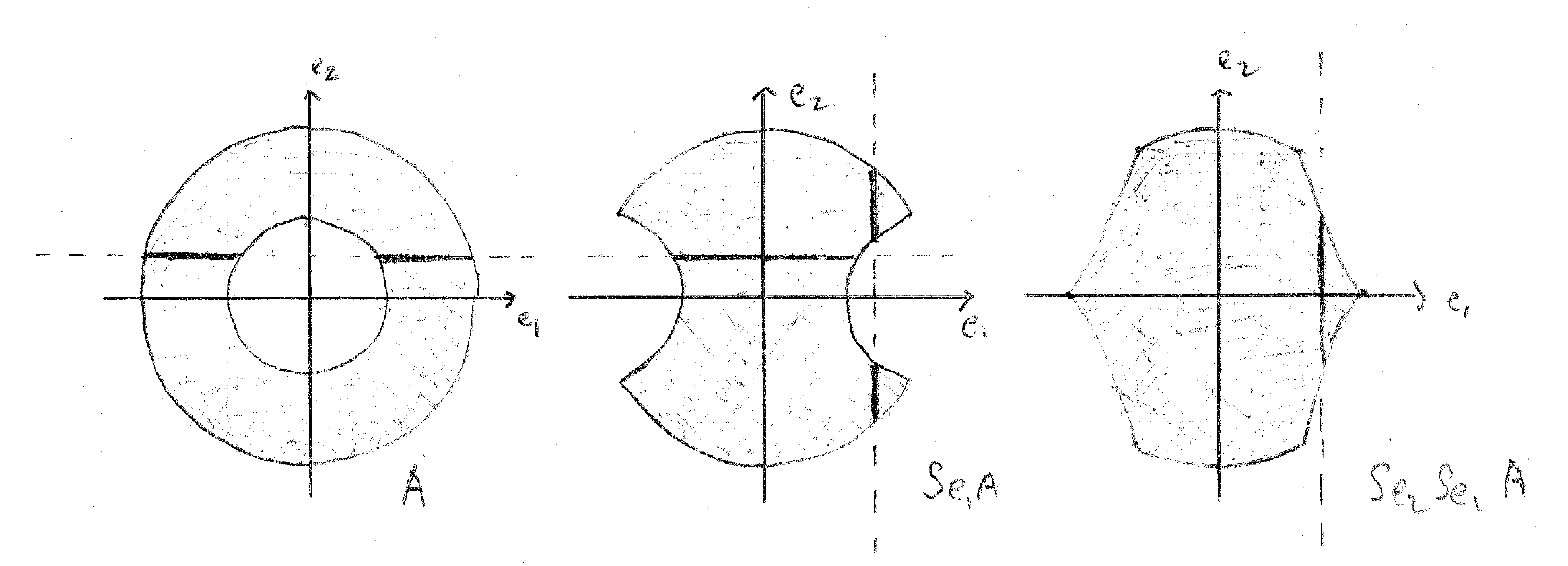}
}
\caption{\small  Steiner symmetrization along the coordinate
directions}
\label{fig:ortho}
\end{figure}

We start with some definitions and notation. 
The dimension $n\ge 2$ will be fixed 
throughout the paper.  The $n$-dimensional 
volume of a Lebesgue measurable set $A\subset \RR^n$ is 
denoted by $\Vol(A)$.  
By $\Per(A)$, we mean the {\bf Caccioppoli perimeter}
of $A$, defined by 
$$
\Per(A) = \sup_{||F||_\infty\le 1}
\int_A {\rm div\,}F(x)\, dx\,,
$$
where the supremum runs over all smooth compactly 
supported vector fields $F$ on $\RR^n$.  If 
$\Per(A)<\infty$, then its value agrees with 
the $(n\!-\!1)$-dimensional Hausdorff measure of the essential 
boundary of $A$. 

We denote by $B$ the closed unit ball in $\RR^n$, 
centered at the origin, and by $\omega_n$ its 
volume.  The closed centered ball of radius $\rho>0$ will
be denoted by  $B_\rho$.  The {\bf Minkowski sum}
of two subsets $A,C\subset\RR^n$ is given by
$$
A+C = \{a+c: a\in A, c\in C\}\,.
$$
Their {\bf Minkowski difference} is the largest set whose 
Minkowski sum with $C$ lies in $A$,
$$
A-C=\{x\in\RR^n: x+c\in A\ \forall c\in C\}\,.
$$
The sets $A+B_\rho$ and $A-B_\rho$ will be called the
{\bf outer} and {\bf inner parallel sets} of $A$.
The {\bf Hausdorff distance} between $A$ and $C$
is given by
$$
d_H(A,C)= \inf\{\delta>0: A\subset C+B_\delta \
\mbox{and}\ C\subset A+B_\delta\}\,.
$$

Let $A\subset \RR^n$ be a compact set, and let
$u\in\RR^n$ be a unit vector.
The {\bf Steiner symmetrization} 
of $A$ in the direction of $u$ is 
defined by the following property.
For each point $x\perp u$,
we compute the length of the intersection of
$A$ with the inverse image of $x$ under 
the orthogonal projection onto the hyperplane $u^\perp$, 
and then replace it with the closed interval 
of the same one-dimensional measure centered on 
$u^\perp$.  If the intersection 
is empty, then the interval is empty; if it is a nonempty set of
measure zero, then the interval consists of a single point.
The resulting set will be denoted by $S_uA$.
Note that $S_uA$ is compact and symmetric under
reflection at $u^\perp$. 
By Cavalieri's principle, $S_uA$ has the same volume as $A$,
i.e., Steiner symmetrization is a volume-preserving 
rearrangement.  The {\bf symmetric rearrangement} of $A$
is the closed centered ball  $A^*$ of the same
volume as $A$.  If $A$ is empty, we take $A^*$ to be empty;
if $A$ is a non-empty set of measure zero, then $A^*=\{0\}$.
We will refer to the radius of $A^*$ as the {\bf volume radius} of $A$.

The corresponding symmetrizations of functions 
are defined by symmetrizing their level sets.
Let $f$ be a nonnegative continuous function with 
compact support. Its Steiner symmetrization $S_uf$ is
determined by the property that 
$$
\{ x: S_uf(x) \ge t\} = S_u \{x: f(x)\ge t\}
$$
for every $t>0$, and its symmetric decreasing 
rearrangement $f^*$ is the unique radially 
decreasing continuous function that is 
equimeasurable to $f$. These symmetrizations
improve the modulus of continuity
and contract distances in the space of continuous functions.  

It is useful to think of Steiner symmetrization
as an operation on the one-dimensional cross sections
$$ 
A(x)= \{t\in\RR: x+tu\in A\}\,
$$
for $x\in u^\perp$.  By definition, 
$$
(S_uA)(x) = (A(x))^*\,,
$$
where $(A(x))^*$ is the one-dimensional symmetric rearrangement of $A(x)$.
Since one-dimensional symmetrization preserves the subset relation,
Stei\-ner symmetrization preserves it as well, and  therefore
\begin{align*}
S_uA\cap S_uC &\ \supset \  S_u(A\cap C)\,,\\
S_uA\cup S_uC &\ \subset \  S_u(A\cup C)\,.
\end{align*}
In particular, the outradius of $S_uA$ is no larger
than the outradius of~$A$. 

Consider a pair of non-empty cross sections $A(x)$ and $C(y)$.
Let $a(x)$ be the leftmost point in $A(x)$, and let
$c(y)$ be the rightmost point in $C(y)$.  Clearly,
$$
A(x)+C(y)\supset (a(x)+C(y)) \cup (A(x)+c(y))\,,
$$ 
with equality when $A(x)$ and $C(y)$
are intervals.  Since the two sets on the right
hand side have only the point
$a(x)+c(y)$ in common, the one-dimensional
measure of $A(x)+C(y)$ is at least as large
as the sum of the measures of $A(x)$ and $C(y)$.
(This is the Brunn-Minkowski inequality
in one dimension).  It follows that
$(A(x))^*+(C(y))^*$ is contained in 
$(A(x)+C(y))^*$,
and therefore
\begin{equation}
\label{eq:M-sum}
S_uA+S_uC \ \subset \ S_u(A+C)\,
\end{equation}
for every pair of compact sets $A,C\subset\RR^n$.
By definition of the Minkowski difference, this in turn
implies that
\begin{equation}
\label{eq:M-diff}
S_uA-S_uC \ \supset \ S_u(A-C)\,.
\end{equation}
In particular, Steiner symmetrization
reduces the volume of outer parallel sets and 
increases the volume of inner parallel sets.  

In the proof of Theorem~\ref{thm:per}, we will
bound the perimeter of $S_{u_n}\dots S_{u_1}A$
in terms of the volume of its parallel sets. 
Specifically, we  will establish Eq.~(\ref{eq:per})
for the {\bf outer Minkowski perimeter}, given by
$$
\PerM^+(A) = \limsup_{\delta\to 0} \frac{1}{\delta}
(\Vol(A+B_\delta)-\Vol(A))\,,
$$
and then argue that $\PerM^+(A)\ge \Per(A)$
for every compact set $A$.
The first lemma concerns the Minkowski sum and difference
of a compact set $A$ with a line segment.

\begin{lem} \label{lem:line}
Let $A\subset\RR^n$ be a compact set,
let $u$ be a unit vector in $\RR^n$, and fix $\beta>0$. Assume
that $S_uA=A$, and consider
the line segment $L_{\beta u}= \{ t u, |t| \leq \beta\}$.
Then, for every $R>0$,
\begin{align*}
\Vol((A+L_{\beta u})
\cap B_R) &\ \le \ \Vol(A\cap B_R) + 2 \omega_{n-1}R^{n-1}\beta\,,\\
\Vol((A-L_{\beta u})\cap B_R) 
          &\ \ge \ \Vol(A\cap B_R) - 2\omega_{n-1}R^{n-1}\beta\,.
\end{align*}
\end{lem}

\begin{proof} 
By assumption, each non-empty
cross section $A(x)$ is either
a centered interval of positive
length $\ell(x)$, or a single point, in which case we set
$\ell(x)=0$. The corresponding cross 
section of $A+L_{\beta u}$ is a line segment of length
$\ell(x)+2\beta$.  The corresponding cross 
section of $A-L_{\beta u}$ is either a centered
interval of length $\ell(x)-2\beta$, a single point, or empty.
The claims follow upon integration over $x\in u^\perp\cap B_R$.
\end{proof}

\smallskip

\begin{lem}
\label{lem:box}
Let $A\subset\RR^n$ be a compact set, and let $R>0$.
For a finite collection of unit vectors 
$u_1,\dots, u_k\in\RR^n$ and $\beta_1,\dots, \beta_k\ge 0$, set
$$
A_k=S_{u_k}\dots S_{u_1} A\,,\qquad
C_k= \sum_{i\le k} S_{u_k} ...S_{u_{i+1}} L_{\beta_i u_i}\,,
$$
where $L_{\beta_i u_i}$ is a line segment as
in Lemma~\ref{lem:line}. Then 
\begin{align*}
\Vol((A_k+C_k)\cap B_R)
&\ \leq \ \Vol(A\cap B_R)+ 2 \omega_{n-1}R^{n-1} 
\sum_{i\le k} \beta_i\,, \\
\Vol((A_k-C_k)\cap B_R)
&\ \geq \ \Vol(A\cap B_R)- 2 \omega_{n-1}R^{n-1} 
\sum_{i\le k} \beta_i\,.
\end{align*}
\end{lem}
\begin{proof}
We proceed by induction on $k$.
Lemma~\ref{lem:line} settles the base case $k=1$.
Let $1<k\le n$, and
suppose the first claim holds for $k-1$.
By Eq.~\eqref{eq:M-sum},
$$
A_k+C_k  \ \subset \ S_{u_k} (A_{k-1}+C_{k-1})  + L_{\beta_ku_k}\,.
$$
We combine this with
the first inequality of Lemma~\ref{lem:line}
and then apply the inductive hypothesis to obtain
\begin{align*}
\Vol((A_k+C_k)\cap B_R) 
&\ \le \ \Vol(S_{u_k} ((A_{k-1}+C_{k-1})\cap B_R) + L_{\beta_ku_k})\\
&\ \le \ \Vol((A_{k-1}+C_{k-1})\cap B_R) + 2\omega_{n-1}R^{n-1}\beta_k\\
&\ \le \ \Vol(A\cap B_R) +2\omega_{n-1}R^{n-1} \sum_{i\le k}\beta_i \,.
\end{align*}
This completes the induction.  For the
second claim, we argue similarly, using
Eq.~\eqref{eq:M-diff} and the second inequality
of Lemma~\ref{lem:line}.
\end{proof}

\smallskip 
The next lemma gives a lower bound
for the inradius of the parallelepiped
$C_n= \sum_{i\le n} S_{u_n}\dots S_{u_{i+1}}L_{\beta u_i}$.

\begin{lem} \label{lem:ball}
Let $u_1,\dots u_n$ be linearly independent
unit vectors in $\RR^n$, and let $\beta,\rho>0$.
If $\beta \det\,(u_1,\dots, u_k)\ge \rho$,
then
$$
B_\rho\ \subset\ 
\sum_{i\le n}S_{u_n}\dots S_{u_{i+1}} L_{\beta u_i}
\,.
$$ 
\end{lem}

\begin{proof} 
Denote by $V_k$ the subspace spanned 
by $u_1,\dots, u_k$, and set
$$
C_k= \sum_{i\le  k} S_{u_k}\dots S_{u_{i+1}} L_{\beta u_i}\,, \qquad k=1,\dots, n\,.
$$
Let $\rho_k$ be the inradius of $C_k$ (considered
as a subset of $V_k$),
and let $\lambda_k$ be the $k$-dimensional 
measure of the parallelepiped spanned
by $u_1,\dots, u_k$.
We will show by induction on $k$ that $\rho_k\ge \beta \lambda_k$
for $k=1,\dots, n$.

In the base case $k=1$, we have $C_1=L_{\beta u_1}$,
$\rho_1=\beta$, and $\lambda_1=1$.
Let now $1<k\le n$, and suppose we have already shown that 
$\rho_{k-1} \ge \beta \lambda_{k-1}$.
By definition,
$$
C_k=S_{u_k}C_{k-1}+ L_{\beta u_k}\,.
$$
The Steiner symmetrization $S_{u_k}$
acts on subsets of $V_{k-1}$ as the orthogonal
projection onto $u_k^\perp$. 
Let $\theta_k$ be the angle between $V_{k-1}$ and
$u_k$.  The projection onto $u_k^\perp$
shrinks the length of vectors in $V_{k-1}$
by a factor that is no smaller than
$\sin \theta_k$, and shrinks the $(k\!-\!1)$-dimensional volume of subsets
exactly by a factor $\sin \theta_k$.
By the inductive assumption,
$$ 
\rho_{k}\ge \rho_{k-1}\sin\theta_{k-1}
\ge \beta \lambda_{k-1}\sin\theta_{k-1}= \beta \lambda_k\,,
$$
completing the induction.
\end{proof}

\begin{thm}[Volume estimate] \label{thm:vol}
If $A\subset\RR^n$ is a compact set with outradius $R$
and $u_1, ..., u_n$ are linearly independent
unit vectors in $\RR^n$, 
then 
\begin{align*} 
\Vol(S_{u_n}\dots S_{u_1}A + B_\delta)
&\ \le \ \Vol(A) +  
\frac{a_n(R+\delta)^{n-1}\delta}{|\det\,(u_1,\dots, u_n)\,|}\,,\\
\Vol(S_{u_n}\dots S_{u_1}A - B_\delta)
&\ \ge \ \Vol(A) - 
\frac{a_n R^{n-1}\delta}{|\det\,(u_1,\dots, u_n)\,|}
\end{align*}
for every $\delta>0$. Here, $a_n=2n\omega_{n-1}$.
\end{thm}

\begin{proof} 
We apply Lemma~\ref{lem:ball} with
$\beta=\delta / |\det\,(u_1,\dots, u_n)\,|$ to see that
$$
B_\delta \ \subset\ \sum_{i\le n} S_{u_n}\dots S_{u_{i+1}}
L_{\beta u_i}=: C\,.
$$
It follows from the first inequality in Lemma~\ref{lem:box} that 
\begin{align*}
\Vol(S_{u_n}\dots S_{u_1}A+B_\delta) 
&\ \le \ \Vol((S_{u_n}\dots S_{u_1}A+C)\cap B_{R+\delta})\\
&\ \le \ \Vol(A) + 2n\omega_{n-1} (R+\delta)^{n-1}\beta\,,
\end{align*}
proving the first claim. Similarly,  we obtain 
from the second inequality in Lemma~\ref{lem:box} 
\begin{align*}
\Vol(S_{u_n}\dots S_{u_1}A-B_\delta) 
&\ \ge \ \Vol((S_{u_n}\dots S_{u_1}A-C)\cap B_R)\\
&\ \ge \ \Vol(A) - 2n\omega_{n-1} R^{n-1}\beta\,.\\[-1.5cm]
\end{align*} 
\end{proof}

\smallskip The next lemma is not needed
for the proof of the main result. It
will be used at the end of the paper
to turn the volume estimate from Theorem~\ref{thm:vol}
into an inequality for the volume radius of parallel sets.

\begin{lem}
\label{lem:rho}
Let $A$ be a non-empty compact set in $\RR^n$ with $n\ge 2$.
For $\delta>0$, let $\rho(\delta)$ be the volume
radius of $A+B_\delta$, let $\rho(-\delta)$
be the volume radius of $A-B_\delta$,
and let $r$ be the volume radius of $A^*$.
Assume that
\begin{align*} \label{eq:vol-sum}
\Vol(A + B_\delta)
&\ \le \ \Vol(A) +  b(R+\delta)^{n-1}\delta\,,\\
\Vol(A - B_\delta)
&\ \ge \ \Vol(A) - b R^{n-1}\delta
\end{align*}
for all $\delta> 0$, where $b\ge 2\omega_n r^n/R^n$ and $R\ge r$ 
are constants. Then
$$
|\rho(\pm\delta)-r|\le c\delta\,,
$$
where $c=bR^{n-1}/(\omega_n r^{n-1})$.
\end{lem}

\begin{proof} Note that $c\ge 2r/R$.
By Jensen's inequality,
\begin{align*}
\Vol(A^*+B_{c\delta})-\Vol(A)
&\ = \ n\omega_n \int_0^{c\delta} (r+s)^{n-1}\, ds\\
&\ \ge \ cn\omega_n r^{n-1} \delta \left(1+\tfrac{c\delta}{2r}\right)^{n-1}\\
&\ \ge \ bR^{n-1} \delta \left(1+\tfrac{\delta}{R}\right)^{n-1}\\
&\ \ge \ \Vol(A+B_\delta)-\Vol(A)\\
&\ =\  \Vol(B_{\rho(\delta)})-\Vol(A)\,,
\end{align*}
where the last two steps used
the assumption on $A+B_\delta$ and the definition of
$\rho(\delta)$. It follows that $A^*+B_{c\delta} \supset B_{\rho(\delta)}$,
which gives the claim for $t>0$.
On the other hand, the assumption on $A-B_\delta$ implies that
$$
\omega_n r^{n-1}(r-\rho(-\delta))
\le \Vol(B_{\rho(-\delta)}) \le bR^{n-1}\delta\,,
$$
which gives the $\rho(-\delta)$.
\end{proof}

\begin{proof}[Proof of Theorem~\ref{thm:per}] 
The first inequality of Theorem~\ref{thm:vol}
yields for the outer Minkowski perimeter
\begin{align*}
\PerM^+(S_{u_n}\dots S_{u_1}A)
&\ =\ \limsup_{\delta\to 0} \frac{1}{\delta}
(\Vol(S_{u_n}\dots S_{u_1}A+B_\delta) -\Vol(A))\\
& \ \le \ \frac{a_nR^{n-1}}
{|\det\,(u_1,\dots, u_n)\,|}\,.
\end{align*}
The proof is completed with the lemma below.
\end{proof}

\smallskip

\begin{lem} \label{lem:perper}
If $A\subset\RR^n$ is compact, then
$\Per(A) \le \PerM^+(A)$.
\end{lem}

\begin{proof}[Proof (L. Ambrosio)] Apply the coarea formula 
(see~\cite[Theorem 13.1]{M}) to the function
$f(x)={\rm dist}\, (x,A)$, which is clearly Lipschitz continuous, and hence
differentiable almost everywhere.
Since $|\nabla f|=1$ a.e. outside $A$
and vanishes a.e. on $A$,
\begin{align*}
\Vol(A+B_\delta) - \Vol(A) 
&\ =\ \int_{A+B_\delta}|\nabla f(x) | \,dx \\
&\ =\ \int_0^\delta \Per(A+B_t)\, dt\\
&\ \ge \   \delta \cdot \inf_{0<t<\delta} \Per(A+B_t)\,.  
\end{align*}
In the second step, we have used the coarea formula
and observed that
$f^{-1}(t)=\partial (A+B_t)$ for $t>0$.
We now divide by $\delta$ and take $\delta\to 0$. Since $A$ is compact,
the parallel set $A+B_\delta$ converges
to $A$ in symmetric difference. It
follows from the lower semicontinuity of the perimeter that
\begin{align*}
\PerM^+(A) &\ = \ \limsup_{\delta\to 0}
\frac{1}{\delta}(\Vol(A+B_\delta) - \Vol(A)) \\
&\ \geq\  \liminf_{\delta\to 0} \Per(A+B_\delta)\\
&\ \geq\  \Per(A)\,.
\end{align*} 
This concludes the proof of the main result.
\end{proof}

\smallskip There are various notions of perimeter,
which agree for open sets with smooth boundary but may differ 
for less regular sets (see~\cite{ACV} for some
recent results).  In particular, $\Per(A)$ can be much 
smaller than the $(n\!-\!1)$-dimensional 
Hausdorff measure of the topological boundary of $A$.  
Having established Eq.~\eqref{eq:per} 
for the Caccioppoli perimeter, we wish to extend the inequality to 
another commonly used measure of the size of the boundary.

The {\bf two-sided Minkowski perimeter} of a compact
set $A$ is defined by
$$
\PerM(A) = \limsup_{\delta\to 0} \frac{1}{2\delta}
\Vol(\partial A+B_\delta)\,,
$$
where $\partial A$ is the boundary of $A$.
It is not hard to show, using a Vitali covering argument,
that the $(n\!-\!1)$-dimensional Hausdorff measure
of $\partial A$ is bounded by $2\cdot 3^{n-1} 
(\omega_{n-1}/\omega_n)\,\PerM(A)$,
but it is not clear (to us) whether the bound holds
without the constant factor.  The last
lemma will be used to relate $\partial A$ to 
the outer and inner parallel sets of~$A$.

\begin{lem} 
\label{lem:dA+C}
For any pair of compact sets $A,C\subset\RR^n$
and every $\delta>0$,
$$
\partial A + C \ \subset \ (A+C)\setminus\,{\rm interior}\,(A-C^-)\,,
$$
where $C^-=\{-c: c\in C\}$ is the reflection of $C$ through
the origin. If $C$ is connected, then the converse
inclusion also holds.
\end{lem}

\begin{proof}
Clearly, $\partial A+C\subset A+C$.
We need to show that $\partial A+C$ does not
intersect the interior of $A-C^-$. 
Suppose that $x$ lies in the interior of $A-C^-$. 
Then there exists $\delta>0$ such that
$B_\delta(x)\subset A-C^-$. This means that
$B_\delta(x-c)\subset A$, i.e., 
$\dist(x-c, \partial A) \ge \delta$
for every $c\in C$. We conclude that
$x$ cannot lie in $\partial A+C$.

For the reverse inclusion, assume furthermore 
that $C$ is connected. Let
$x\in (A+C)\setminus {\rm interior}\, (A-C^-)$,
and consider
\begin{align*}
C_1&=\{c\in C: x-c\in A\}\,,\\
C_2&=\{c\in C: x-c\not \in \, {\rm interior}\, (A)\}\,.
\end{align*}
By definition, $C_1$ and $C_2$ are closed and cover $C$.
Furthermore, $C_1$ is non-empty because $x\in A+C$,
and $C_2$ is non-empty because
$x\not\in {\rm interior}\,(A-C^-)$.
Since $C$ is connected, $C_1$ and $C_2$ cannot be disjoint.
Pick $c\in C_1\cap C_2$. Then $x-c\in \partial A$, 
i.e., $x\in\partial A+C^-$, as claimed.
\end{proof}

In the special case where $C=B_\delta$, the lemma implies that
\begin{equation}\label{eq:dA+B}
\Vol(\partial A + B_\delta)
=\Vol(A+B_\delta) - \Vol(A-B_\delta)\,,
\end{equation}
because the boundary of $A-B_\delta$, which consists of all points
having distance exactly $\delta$ from
the complement of $A$, is a set of volume zero.
Combining Eq.~\eqref{eq:dA+B}
with Theorem~\ref{thm:vol}, we obtain
\begin{align*}
\Vol(\partial S_{u_n}\dots S_{u_1}A + B_\delta)
\hspace{-2cm}\\
&=\Vol(S_{u_n}\dots S_{u_1}A+B_\delta) - 
\Vol(S_{u_n}\dots S_{u_1}A-B_\delta)\\
&\le \frac{4n\omega_{n-1}(R+\delta)^{n-1}\delta}
{|\det\,(u_1,\dots, u_n)\,|}\,.
\end{align*}
Dividing by $2\delta$ and taking $\delta\to 0$
extends Eq.~\eqref{eq:per}
to the two-sided Minkowski perimeter.

\begin{cor} Under the assumptions of Theorem~\ref{thm:per},
$$
\PerM(S_{u_n}\dots S_{u_1}A) \ \le \
\frac{a_nR^{n-1}}{|\det\,(u_1,\dots, u_n)\,|}\,.
$$
\end{cor}

\noindent Since $\Per(A)\le \PerM(A)$ by the same 
reasoning as in Lemma~\ref{lem:perper},
this improves upon Theorem~\ref{thm:per}.

\medskip 
Finally, we discuss an application to random sequences 
of Steiner symmetrizations. 
Consider a non-empty compact set $A\subset\RR^n$,
let $r$ be its volume radius, and assume that $A\subset B_R$.
Let $\{U_k\}_{k\ge 0}$ be a sequence of
unit vectors picked independently, uniformly at 
random from the unit sphere in $\RR^n$,
and define recursively
$$
A_0=A\,,\qquad A_{k+1}=S_{U_k} A_k \quad (k\ge 0)\,.
$$
It was recently shown by Burchard and Fortier
that the expectation of the symmetric difference 
from $A_k$ to $A^*$ satisfies
\begin{equation}\label{eq:BF-rate-A}
E(A_k\bigtriangleup  A^*)
\le n\omega_n 2^{n+1}R^n k^{-1}
\end{equation}
for all $k>0$~\cite[Proposition 5.2]{BF}. 
Under suitable regularity assumptions on $\partial A$, 
this can be used to bound the
Hausdorff distance $d_H(\partial A_k,\partial A^*)$,
which controls how much the outradius and inradius
of $A_k$ differ from its volume radius.

We briefly describe the 
tools developed in~\cite[Section 7]{BF}. The authors
consider the auxiliary function
$$
f(x)=\dist(x,\RR^n\setminus A) + (R-\dist(x,A))_+
$$
and its symmetrizations 
$$
F_0=f\,,\qquad F_{k+1}=S_{U_k}F_k\quad (k\ge 0)\,.
$$
By construction, $A_k=\{x: F_k(x)>R\}$.
Using Eqs.~\eqref{eq:M-sum} and~\eqref{eq:M-diff},
they show that
\begin{equation} \label{eq:BF-dH}
d_H(\partial A_k,\partial A^*) \le 
\max_\pm \bigl|\rho(\pm ||F_k-f^*||_\infty)-r\bigr|\,,
\end{equation}
where $\rho(\pm\delta)$ is the volume radius of the 
parallel set $A\pm B_\delta$.
It follows from~\cite[Proposition 5.2]{BF} that
\begin{equation}\label{eq:BF-rate}
E(||F_k-f^*||_\infty)\le 12Rk^{-\frac{1}{n+1}}
\end{equation}
for $k>0$. 
Under the assumption that $A$ has finite Minkowski perimeter, 
they differentiate $\rho$ at $\delta=0$
and obtain from Eqs.~\eqref{eq:BF-dH}
and~\eqref{eq:BF-rate} a sequence
of Steiner symmetrizations along which
$d_H(\partial A_k,\partial A^*) = O(k^{-\frac{1}{n+1}})$ 
as $k\to\infty$.  The rate of convergence estimates in 
Eqs.~\eqref{eq:BF-rate-A} and~\eqref{eq:BF-rate}
are proved by comparing Steiner symmetrization with 
polarization, a simpler rearrangement that preserves 
perimeter as well as volume~\cite[Section 5]{BF}.

We will use Theorems~\ref{thm:per}
and~\ref{thm:vol} to obtain a stronger bound 
on the perimeter of $A_{n+1}$
that results in stronger bounds
on $\rho$ and, through Eq.~\eqref{eq:BF-dH},
on $d_H(\partial A_k,\partial A^*)$.
By Theorem~\ref{thm:per}, 
the perimeter of $A_n$ is almost surely finite,
because the probability that the vectors 
$U_0,\dots, U_{n-1}$ lie in a common hyperplane is
zero.  We next argue that $\Per(A_{n+1})$ has finite
expectation.
Since $\Per(A_{n+1})\le \Per(A_n)$, we can apply
Theorem~\ref{thm:per} to $A_n$ and $A_{n+1}$
to obtain
$$
\Per(A_{n+1})\le a_nR^{n-1} Y_n\,,
$$
where $a_n=2n\omega_{n-1}$, and the random variable $Y_n$ is given by
\begin{equation} \label{eq:def-Y}
Y_n=
\min \left\{ |\det \,(U_1,\dots, U_n)\,|^{-1},
|\det \, (U_0,\dots, U_{n-1})\,|^{-1}\right\}\,.
\end{equation}
As in the proof of Lemma~\ref{lem:ball}, we expand 
$ |\det\,(U_1,\dots U_n)\,|=\prod_{k=2}^n X_k$
and $|\det\,(U_0,\dots U_{n-1})\,|=X_n'\cdot \prod_{k=2}^{n-1} X_k$,
where $X_k$ is the Euclidean distance
of $U_k$ to the subspace of $\RR^n$
spanned by $U_1,\dots U_{k-1}$, and $X_n'$ 
is the distance of $U_0$ to the subspace
spanned by $U_1,\dots, U_{n-1}$. 
Then
$$
Y_n= (\max\{X_n,X_n'\})^{-1}\cdot \prod_{k=2}^{n-1} X_k^{-1}\,.
$$
By rotational invariance, $X_k$ has the same distribution
as the distance of a random point on the 
sphere from a $(k\!-\!1)$-dimensional
coordinate plane,
$X_n'$ has the same distribution as $X_n$,
and $X_2,\dots, X_n, X_n'$ are independent. 
Since the sphere is compact and intersects the coordinate
planes transversally, there exist constants
$b_{n,k}$ such that 
$P(X_k\le t)\le b_{n,k}t^{n-k+1}$ for $2\le k\le n$.
By the independence of $X_n$ and $X_n'$,
it follows that
$P(\max\{X_n,X_n'\}\le t) \le (b_{n,n} t)^2$.
Therefore,
\begin{equation} \label{eq:Yn}
E(Y_n)=E((\max\{X_n,X_n'\})^{-1})
\cdot \prod_{k=2}^{n-1} E(X_k^{-1}) <\infty\,.
\end{equation}
We have proved the following inequality.

\begin{cor} 
\label{cor:n+1}
$$
E(\Per(A_{n+1}))\ \le\ b_n R^{n-1}\,,
$$
where $b_n=2n\omega_{n-1}E(Y_n)$ depends only on the dimension.
\end{cor}

\smallskip We want to apply Eqs.~\eqref{eq:BF-dH} and~\eqref{eq:BF-rate}
to the conditional expectation $E(\cdot\  \vert\ A_{n+1})$.
Let $\tilde f$, $\tilde F_k$, and $\tilde\rho$
be the functions corresponding to $f$, $F_k$, and $\rho$ 
with $A_{n+1}$ in place of $A$. 
Replacing Theorem~\ref{thm:per} with Theorem~\ref{thm:vol}
in the proof of Corollary~\ref{cor:n+1}, we obtain for every $\delta>0$,
\begin{align*} 
\Vol(A_{n+1}+B_\delta)
&\le \Vol(A) + a_n (R+\delta)^{n-1}\delta Y_n\,,\\
\Vol(A_{n+1}+B_\delta)
&\ge \Vol(A) - a_n R^{n-1}\delta Y_n\,,
\end{align*}
where $a_n=2n\omega_{n-1}$,
and $Y_n$ is the random variable 
from Eq.~\eqref{eq:def-Y}.  Since 
$a_n\ge 2\omega_n$, $Y_n\ge 1$,
and  $R\ge r$, 
the assumptions of Lemma~\ref{lem:rho} are satisfied
with $b=a_nY_n$, and so
$$
|\tilde\rho(\pm \delta)-r|
\le \frac{a_nR^{n-1}}{\omega_n r^{n-1}}\delta\,.
$$
It follows that
\begin{align*}
d_H(\partial A_{n+1+k},\partial A^*) 
&\  = \ \max_{\pm} \bigl |\tilde\rho(\pm ||\tilde F_k-f^*||_\infty) -r\bigr|\\
& \ \le \ \frac{a_nR^{n-1}}{\omega_n r^{n-1}}Y_n\,||\tilde F_k-f^*||_\infty
\end{align*}
for $k>0$, see Eq.~\eqref{eq:BF-dH}.
Since $Y_n$ is independent of $U_k$ for $k>n$
and $\tilde F_k$ depends on $U_0,\dots, U_n$
only through $A_{n+1}$, we can invoke the Markov property
to obtain 
\begin{align*}
E(d_H(\partial A_{n+1+k},\partial A^*)) 
&\ =\ E(E(d_H(\partial A_{n+1+k},\partial A^*)\ \vert\  U_0,\dots, U_n))\\
&\ \le\ \frac{a_nR^{n-1}}{\omega_n r^{n-1}}
E( Y_n \, E(||\tilde F_k-f^*||_\infty \ \vert\ A_{n+1}))\\
&\ \le\ \frac{12a_n R^{n-1}}{\omega_n r^{n-1}} \, E(Y_n)
\, R \, k^{-\frac{1}{n+1}}
\end{align*}
for $k>0$.  In the last line, we have 
applied Eq.~\eqref{eq:BF-rate} to 
$\tilde F_k$.  By Corollary~\ref{cor:n+1}, the expected value
of $Y_n$ is finite. We shift the index and adjust the constant
to obtain the desired bound on the rate of convergence.

\begin{cor}\label{cor:rate} 
$$ E(d_H(\partial A_k,\partial A^*))
\ \le \ c_n \, (R/r)^{n-1} R \, k^{-\frac{1}{n+1}}\,,
$$
where $c_n=25 n \omega_{n-1} E(Y_n)/\omega_n$.
\end{cor}

\smallskip We close with an explicit
bound on the constants $b_n$ and $c_n$
that appear in Corollaries~\ref{cor:n+1} and~\ref{cor:rate}.
We consider separately each of the 
expected values in Eq.~\eqref{eq:Yn}. A routine
spherical integral (conveniently evaluated as
a Gaussian integral 
over $\RR^n$) gives
$E(X_k^{-1})=\frac{(n-k+1)\omega_{n-k+1}}{(n-k)\omega_{n-k}}
\cdot \frac{(n-1)\omega_{n-1}} {n\omega_n}$ for $2\le k<n$.
Collecting terms, we obtain 
$$
\prod_{k=2}^{n-1}
E(X_k^{-1}) \ = \ \frac{ ((n\!-\!1)\omega_{n-1})^{n-1}}{2(n\omega_n)^{n-2}}\,.
$$
A similar integral yields
$ P(X_n \le \sin\alpha) =\frac{(n-1)\omega_{n-1}}{n\omega_n}
\int_{-\alpha}^\alpha
(\cos t)^{n-2}\, dt$.
Using that $X_n$ and $X_n'$ are independent, we estimate 
for $n\ge 3$
$$
E((\max\{X_n,X_n'\})^{-1})  \ \le \ 
1+ \left(\frac{2(n\!-\!1)\omega_{n-1}}{n\omega_n}\right)^2\,.
$$
When $n=3$, this equation holds with equality, 
resulting in $E(Y_3)=\pi$.  In two dimensions, we find that $E(Y_2)\le 2$,
and for $n\to\infty$, we have $\lim n^{-1}\log E(Y_n)=\sqrt{2e}$.
Since $\lim n^{-1}\log (n\omega_{n-1})=-\infty$
and $\lim n^{-1} \log (n\omega_{n-1}/\omega_n)=1$,
we conclude that $b_n$ converges to zero and $c_n$ grows exponentially.

\bigskip\noindent {\bf Acknowledgments.}
This work was partially supported by 
an NSERC Discovery Grant (Burchard) and
an NSERC Canadian Graham Bell Graduate Scholarship (Chambers). 
We would also like to thank 
Luigi Ambrosio for the proof of Lemma \ref{lem:perper}.

\bibliographystyle{amsplain}

\begin{thebibliography}{1}

\bibitem{ACV}
L.~Ambrosio, A.~Colesanti, and E.~Villa, \emph{Outer {M}inkowski content for
  some classes of closed sets}, Mathematische Annalen \textbf{342} (2008),
  no.~4, 727--748.

\bibitem{BF}
A.~Burchard and M.~Fortier, \emph{Random polarizations}, submitted for
  publication (2011). Preprint {\tt arXiv:1104.4103}.

\bibitem{CCF}
M.~Chleb\'{i}k, A.~Cianchi, and N.~Fusco, \emph{The perimeter inequality under
  {Steiner} symmetrization: cases of equality}, Ann. of Math \textbf{162}
  (2005), 525--555.

\bibitem{M}
F.~Maggi, \emph{Sets of {F}inite {P}erimeter and {G}eometric {V}ariational
  {P}roblems}, Cambridge University Press, 2012.

\bibitem{S}
J.~Steiner, \emph{Einfacher {Beweis} der isoperimetrischen {Haupts\"atze}}, J.
  Reine Angew. Math (Crelle's Journal) \textbf{18} (1838), 281--296.

\end{thebibliography}
\providecommand{\bysame}{\leavevmode\hbox to3em{\hrulefill}\thinspace}
\providecommand{\MR}{\relax\ifhmode\unskip\space\fi MR }
\providecommand{\MRhref}[2]{%
  \href{http://www.ams.org/mathscinet-getitem?mr=#1}{#2}
}
\providecommand{\href}[2]{#2}

\end{document}